\pdfoutput=1
\documentclass[11pt]{amsart}

\usepackage[
text={450pt,575pt},
headheight=9pt,
centering
]{geometry}

\usepackage[english]{babel}
\usepackage[utf8]{inputenc}
\usepackage{paralist,url,verbatim}
\usepackage{amscd}
\usepackage[T1]{fontenc}
\usepackage{xspace}
\usepackage{comment}
\usepackage{amsmath,amsthm,amssymb,amsfonts}
\usepackage{mathtools}
\usepackage{enumitem}
\usepackage{wasysym}
\usepackage{color}
\usepackage[bookmarksopen=true]{hyperref}
\definecolor{darkblue}{RGB}{0,0,160}
\hypersetup{
colorlinks,%
citecolor=black,%
filecolor=black,%
linkcolor=darkblue,%
urlcolor=darkblue
}

\usepackage{amsthm}

\newtheorem{thm}{Theorem}[section]

\newtheorem{cor}[thm]{Corollary}
\newtheorem{prop}[thm]{Proposition}

\theoremstyle{definition}
\newtheorem{defn}[thm]{Definition}

\theoremstyle{remark}
\newtheorem{example}[thm]{Example}
\newtheorem{remark}[thm]{Remark}

\newcommand\Perp{\protect\mathpalette{\protect\independenT}{\perp}}
\def\independenT#1#2{\mathrel{\rlap{$#1#2$}\mkern2mu{#1#2}}}
\newcommand{\ind}[2]{\left.#1 \Perp #2 \inD}
\newcommand{\inD}[1][\relax]{\def\argone{#1}\def\temprelax{\relax}
  \ifx\argone\temprelax\right.\else\,\middle|#1\right.{}\fi}

\begin{document}

\title{Algebraic Aspects of Conditional Independence and Graphical Models}

\author[T.~Kahle]{Thomas Kahle} \address {Fakultät für Mathematik,
  Otto-von-Guericke University Magdeburg, Germany}
\urladdr{\url{http://www.thomas-kahle.de}}

\author[J.~Rauh]{Johannes Rauh}
\address{Max-Planck-Institute for Mathematics in the Sciences,
Leipzig, Germany.}
 \urladdr{\url{http://www.yorku.ca/jarauh/}}

\author[S.~Sullivant]{Seth Sullivant}
\address{Department of Mathematics,
     North Carolina State University, Raleigh, USA.}
\urladdr{\url{http://www4.ncsu.edu/~smsulli2/}}

\subjclass[2010]{62-00, 13P25}

\begin{abstract}
This chapter of the forthcoming \emph{Handbook of Graphical Models}
contains an over\-view of basic theorems and techniques from algebraic
geometry and how they can be applied to the study of conditional
independence and graphical models.  It also introduces binomial ideals
and some ideas from real algebraic geometry.  When random variables
are discrete or Gaussian, tools from computational algebraic geometry
can be used to understand implications between conditional
independence statements.  This is accomplished by computing primary
decompositions of conditional independence ideals.  As examples the
chapter presents in detail the graphical model of a four cycle and the
intersection axiom, a certain implication of conditional independence
statements.  Another important problem in the area is to determine all
constraints on a graphical model, for example, equations determined by
trek separation.  The full set of equality constraints can be
determined by computing the model’s vanishing ideal.  The chapter
illustrates these techniques and ideas with examples from the
literature and provides references for further reading.
\end{abstract}

\date{May 2017}

\maketitle

\section{Introduction}

Consider a finite set of random variables~$X_v$, $v\in V$.
Section~1.6 of Part~I\footnote{All references to other sections refer
to the forthcoming \emph{Handbook of Graphical Models}, edited by
Mathias Drton, Steffen Lauritzen, Marloes Maathuis, and Martin
Wainwright.  It will contain this document as Section~3 of Part~I.}
describes how to use a simple undirected graph $G=(V, E)$ to encode
conditional independence (CI) statements among the random variables.
One can also naturally associate a parametrized family of joint
probability distributions of the~$X_v$ to a graph.  For undirected
graphs, the Hammersley--Clifford theorem (see Section~1.6.3 of Part~I
shows that both the implicit method and the parametric method lead to
the same families of probability distributions (called graphical
models), as long as all distributions are assumed strictly positive.

When probabilities are allowed to go to zero, the models defined by
the collections of CI statements contain probability distributions
that do not lie in the parametric graphical model, which, by
definition, consists of strictly positive probability distributions.
In fact, these additional distributions do not even lie in the closure
of the parametric graphical model, so they cannot be approximated by
distributions from the parametric graphical model.  Moreover, models
defined by the defined collections of CI statements (pairwise Markov
properties, local Markov properties, global Markov properties) differ
from one another.
As an example, consider the four-cycle~$C_4$.
\begin{prop}
\label{prop:binary-C4}
The binary random variables $X = (X_1,X_2,X_3,X_4)$ satisfy the global
Markov statements of~$C_4$, $\ind{1}{3}[\{2,4\}]$ and
$\ind{2}{4}[\{1,3\}]$, if and only if one (or more) of the following
statements is true:
\begin{enumerate}
\item The joint distribution lies in the closure of the graphical
  model.
\item There is a pair $(X_i,X_{i+1})$ of neighboring nodes such that
  $X_i=X_{i+1}$ a.s.
\item There is a pair $(X_i,X_{i+1})$ of neighboring nodes such that
  $X_i\neq X_{i+1}$ a.s.
\end{enumerate}
\end{prop}

This chapter shows how to prove results such as
Proposition~\ref{prop:binary-C4} using algebraic tools.
The algebraic method can also be used to study implications between
conditional independence statements.  Here is an example:

\begin{prop}
\label{prop:example}
Suppose that $X,Y,Z$ are binary random variables or jointly normal
random variables.  If $\ind{X}{Y}$ and $\ind{X}{Y}[Z]$ then either
$\ind{X}{(Y,Z)}$ or $\ind{(X,Z)}{Y}$.
\end{prop}

The CI implication in Proposition~\ref{prop:example} is a special case
of the \emph{gaussoid axiom}~\cite{Lnenicka2007}.  One may wonder what
is special about jointly normal or binary random variables.  For
instance, is there a variant of this implication when $X,Y,Z$ are
discrete but not binary?  How can one systematically find and study
implications like this?

CI implications can also be interpreted as intersections of graphical
models.  For example, the two CI statements $\ind{X}{Y}$ and
$\ind{X}{Y}[Z]$ in Proposition~\ref{prop:example} correspond to the
two graphical models $X \to Z \leftarrow Y$ and $X{-Z}{-Y}$,
respectively.  Thus, Proposition~\ref{prop:example} says that the
intersection of these two graphical models equals the union of the two
graphical models $X\ Z{-Y}$ and $X{-Z}\ Y$, provided the random
variables are either binary or jointly normal.  As the example shows,
the intersection of two graphical models need not be a graphical
model.  How can one compute this intersection?

The goal of this chapter is to explore these questions and introduce
tools from computational algebra for studying them.  Our perspective
is that, for a fixed type of random variable, the set of distributions
that satisfy a collection of independence constraints is the zero set
of a collection of polynomial equations.  Solutions of systems of
polynomial equations are the objects of study of algebraic geometry,
and so tools from algebra can be brought to bear on the problem.  The
next section contains an overview of basic ideas in algebraic geometry
which are useful for the study of conditional independence structures
and graphical models.  In particular, it introduces algebraic
varieties, polynomial ideals, and primary decomposition.
Section~\ref{sec:ciideals} introduces the ideals associated to
families of conditional independence statements, and explains how to
apply the basic techniques to deduce conditional independence
implications.  Section~\ref{sec:examples} illustrates the main ideas
with some deeper examples coming from the literature.
Section~\ref{sec:vanishing} concerns the vanishing ideal of a
graphical model, which is a complete set of implicit restrictions for
that model.  This set of restrictions is usually much larger than the
set of conditional independence constraints that come from the graph,
but it can illuminate the structure of the model especially with more
complex families of models involving mixed graphs or hidden random
variables.  Section~\ref{sec:further} highlights some key references
in this area.


\section{Notions of Algebraic Geometry and Commutative Algebra}\label{sec:alggeom}

Commutative algebra is the study of systems of polynomial
equations, and algebraic geometry is the study of geometric properties of their solutions.
Both are rich fields with many deep results.  This section only gives
a very coarse introduction to the basic facts that hopefully makes it
possible for the reader to understand the phenomena and algorithms
discussed in later parts of this chapter.
For a more detailed introduction, the reader is referred to the
standard textbook~\cite{Cox2015}.

\subsection{Polynomials, ideals and varieties}
\label{sec:polyn-ideals-vari}

Let $\Bbbk$ be a field, for example the rational numbers $\mathbb{Q}$,
the real numbers~$\mathbb{R}$, or the complex numbers~$\mathbb{C}$.
Let $\mathbb{N} = \{0,1, \ldots\}$ denote the natural numbers.
Let $p_1, p_2, \dots, p_r$ be a collection of \emph{indeterminates} or
\emph{variables}.  A \emph{monomial} in the indeterminates
$p_1, p_2, \dots, p_r$ is an expression of the form
$p_1^{u_1} p_2^{u_2} \cdots p_r^{u_r}$ where $u_1, \ldots, u_r$ are
nonnegative integers.  Writing $u = (u_1, \ldots, u_r)$, let
\[
p^u  := p_1^{u_1} p_2^{u_2} \cdots p_r^{u_r}.
\]
A \emph{polynomial} is a finite linear combination of monomials, i.e.
\[
f =   \sum_{u \in U}  c_u  p^u
\]
where $U \subset \mathbb{N}^r$ is a finite set and $c_{u} \in \Bbbk$.  Of course,
one is used to thinking of a polynomial as a function,
$f: \Bbbk^r \rightarrow \Bbbk$, which can be evaluated in a point
$a \in \Bbbk^r$ for~$p$.  In the following, this function will usually
have the role of a constraint; i.e., the object of interest is the
zero set $\{a\in\Bbbk^r : f(a) = 0\}$.
In algebra, it is also useful to think of a
polynomial as a formal object, i.e.~the indeterminates are simply
symbols that are used in manipulations, with no need for them to be
evaluated.

The set of all polynomials in indeterminates $p_1, \ldots, p_r$ with
coefficients in $\Bbbk$ is called the \emph{polynomial ring}
and denoted~$\Bbbk[p_1, \ldots, p_r]$.  The word \emph{ring} means
that $\Bbbk[p_1, \ldots, p_r]$ has two operations, namely addition of
polynomials and multiplication of polynomials, and that these
operations satisfy all the usual properties of addition and
multiplication (associativity, commutativity, distributivity).
However, multiplicative inverses need not exist: The result of
dividing one polynomial by another non-constant polynomial is in
general not a polynomial, but a rational function.

\begin{defn}
Let $\mathcal{F} \subseteq \Bbbk[p_1, \ldots, p_r]$.  The
\emph{variety} defined by $\mathcal{F}$ is the vanishing set of the
polynomials in $\mathcal{F}$, that is,
\[
V(\mathcal{F}) = \{ a \in \Bbbk^r :  f(a) = 0  \mbox{ for all } f \in \mathcal{F} \}.
\]
\end{defn}

\begin{example}
  Let $r =2$ and consider
  $\mathcal{F} = \{ x^2 - y \} \subseteq \Bbbk[x,y]$. The variety
  $V( \{ x^2 - y \}) \subseteq \Bbbk^2$ is the familiar parabola
  ``$y = x^2$'' in the plane.  For $r = 4$ and
  $\mathcal{F} = \{ p_{11}p_{22} - p_{12}p_{21} \} \subseteq
  \Bbbk[p_{11}, p_{12}, p_{21}, p_{22}]$, the variety
  $V(\{ p_{11}p_{22} - p_{12}p_{21} \} ) \subseteq \Bbbk^4$ is the set
  of all singular $2 \times 2$ matrices
  $\left(\begin{smallmatrix} p_{11} & p_{12} \\ p_{21} & p_{22}
\end{smallmatrix}\right)$
 with entries in $\Bbbk$.
\end{example}

\begin{example}\label{ex:4points}
Let $\mathcal{F} =  \{ x^2 - y, x^2 + y^2 - 1  \}$.  The variety $V(\mathcal{F})$
is the set of points 
\[
\{(x,y) \in \Bbbk^2 :  y = x^2 \mbox{ and } x^2 + y^2  = 1 \},
\]
in other words, the intersection of a parabola and a circle. The
number of points in this intersection varies depending on whether the
underlying field is $\mathbb{Q}$, $\mathbb{R}$, or $\mathbb{C}$ (or
some other field).  If $\Bbbk = \mathbb{Q}$ the variety is empty, if
$\Bbbk = \mathbb{R}$ the variety has two points, and if $\Bbbk =
\mathbb{C}$, the variety has four points.  In statistical applications
one is usually interested in solutions over~$\mathbb{R}$.  However, it
is often easier to first perform computations in algebraically closed
fields like~$\mathbb{C}$ before restricting to the real numbers, at
least in theory.  On the other hand, when using a computer algebra
system, it may be advantageous to work with $\mathbb{Q}$, if possible,
because rational numbers can be represented exactly on a computer.
\end{example}

The examples so far have always used finite sets~$\mathcal{F}$.  This
is not necessary for the definition of a variety, and it is often
worthwhile to consider the variety $V(\mathcal{F})$ where
$\mathcal{F}$ is an infinite set of polynomials.  In fact, it is often
convenient to replace the original set of polynomials $\mathcal{F}$ by
an infinite set, the \emph{ideal} generated by $\mathcal{F}$, which is
equivalent to $\mathcal{F}$ in some sense but has more structure.

One reason is that different families of polynomials may have the same
varieties.  For example, if $f,g\in\mathcal{F}$, then the variety of
$\mathcal{F}\cup\{f+g\}$ equals $V(\mathcal{F})$.  Similarly, for
$f\in\mathcal{F}$ and $\lambda\in k$, the variety of
$\mathcal{F}\cup\{\lambda f\}$ equals $V(\mathcal{F})$.

\begin{defn}
A set $I \subseteq \Bbbk[p_1, \ldots, p_r]$ is an \emph{ideal} if for
all $f,g \in I$, $f + g \in I$ and for all $f \in I$ and
$h \in \Bbbk[p_1, \ldots, p_r]$,  $hf \in I$.
\end{defn}

\begin{defn}\label{d:idealgen}
Let $\mathcal{F} \subseteq \Bbbk[p_1, \ldots, p_r]$ be a set of
polynomials.  The \emph{ideal generated by} $\mathcal{F}$ is the
smallest ideal in $\Bbbk[p_1, \ldots, p_r]$ that contains
$\mathcal{F}$.  Equivalently, the ideal generated by $\mathcal{F}$
consists of all polynomials $h_1 f_1 + \cdots + h_k f_k$ for
$h_1, \ldots, h_k \in \Bbbk[p_1, \ldots, p_r]$,
$f_1, \ldots, f_k \in \mathcal{F}$, and any $k$.  The ideal generated
by $\mathcal{F}$ is denoted $\langle \mathcal{F} \rangle$.
\end{defn}

\begin{prop}\label{prop:idealgen}
Let $\mathcal{F} \subseteq \Bbbk[p_1, \ldots, p_r]$ be a set of
polynomials.  Then $V(\mathcal{F}) = V( \langle \mathcal{F} \rangle)$.
\end{prop}

\begin{example}
The ideal $\langle x^2 - y, x^2 + y^2 - 1 \rangle$ generated by the
set $\mathcal{F}$ from Example~\ref{ex:4points} has many different
possible generating sets.  For example, an alternate generating set is
$\{ x^2 - y, x^4 + x^2 - 1 \}$.  This allows to easily find all the
solutions of the polynomial system because all roots of the univariate
polynomial $x^4 + x^2 - 1 = 0$ can be plugged into the second
polynomial $x^2 - y = 0$, which can then be solved for~$y$.
\end{example}

Hilbert's basis theorem implies that for any ideal
$I\subseteq \Bbbk[p_1, \ldots, p_r]$ there exists a finite set of
polynomials $\mathcal{F} \subseteq \Bbbk[p_1, \ldots, p_r]$ such that
$I = \langle \mathcal{F} \rangle$.

Even though it is, for theoretical considerations, often easier to
think about systems of polynomial equations in terms of ideals, in
practice (i.e. when working with computer algebra systems), the ideal
is almost always specified in terms of a finite set of generators (or
such a finite set of generators has to be computed on the way).  On
the other hand, during a computation it is often necessary to replace
this set of generators by a more convenient set of generators (e.g. a
Gr\"obner basis), so the generators may change even though the ideal
stays the same along a computation.

\begin{defn}\label{def:vanishing}
  Let $S \subseteq \Bbbk^r$.  The \emph{vanishing ideal} of $S$ is the
  set
  \[
    I(S) :=  \{  f \in \Bbbk[p_1, \ldots, p_r] :  f(a) = 0  \mbox{ for all } a \in S \}. 
  \]
\end{defn}

It is easy to check that a vanishing ideal is indeed an ideal.
Clearly, any ideal~$J$ satisfies $J\subseteq I(V(J))$.  However, the
converse inclusion does not hold in general.  For instance,
$I(V(\langle x^2 \rangle)) = \langle x \rangle$, and
$I(V(\langle x^2y, xy^2 \rangle)) = \langle xy \rangle$ (over any
field $\Bbbk$).

\begin{defn}
The ideal $I(V(J))$ is the \emph{$\Bbbk$-radical of $J$}.
An ideal $J$ such that $I(V(J)) = J$ is a \emph{$\Bbbk$-radical
ideal}.  If $\Bbbk$ is algebraically closed (e.g.~if $\Bbbk =
\mathbb{C}$), such an ideal $J$ is simply called a \emph{radical
ideal}.
\end{defn}

Radical ideals can also be characterized algebraically, and there are
algorithms to compute radicals.  The radical is usually a simpler
ideal, and if the radical of an ideal can be computed, it is
advantageous to do this in a first step in each calculation (as long
as one is only interested in properties of~$V(J)$, and not in
algebraic properties of~$J$).

The following proposition illustrates the close relation between
ideals and varieties.

\begin{prop}\label{prop:cupcap} $ $
\begin{itemize}
\item Let $S_1,S_2\subseteq\Bbbk^r$.  Then $I(S_1\cup S_2) = I(S_1)
\cap I(S_2)$ and $I(S_1\cap S_2)  \supseteq I(S_1) + I(S_2) := \{ f + g: f\in
I(S_1), g\in I(S_2)\}$.
\item Let $I,J\subseteq\Bbbk[p_1,\dots,p_r]$.  Then 
$V(I\cup J) = V(I + J) =  V(I)\cap V(J)$ and $V(I\cap J) = V(I)\cup V(J)$.
\end{itemize}
\end{prop}

\subsection{Irreducible and primary decomposition}
\label{sec:decompositions}

Proposition~\ref{prop:cupcap} shows that the union of two varieties is
again a variety.  Interestingly, not every variety can be written as a
non-trivial finite union.

\begin{defn}
  A variety $V$ is \emph{reducible} if there are two varieties
  $V_1, V_2 \neq V$ such that $V_1 \cup V_2 = V$.  Otherwise $V$ is
  \emph{irreducible}.
\end{defn}

\begin{thm}
Any variety $V \subseteq \Bbbk^r$ has a unique decomposition into
finitely many irreducible varieties $V = V_1 \cup \cdots \cup V_k$
(with $V_i\not\subseteq V_j$ for $i\neq j$).
\end{thm}
The varieties $V_1,\dots,V_k$ are called the \emph{irreducible
components} of~$V$.

\begin{thm}
\begin{enumerate}
\item Let $V$ be an irreducible variety, and let
$\phi:\Bbbk^r\to\Bbbk^s$ be a rational map.  Then $V(I(\phi(V)))$ is
irreducible.
\item Let $V$ be a variety that has a rational
parametrization $\phi:\Bbbk^r\to V$ such that the image of $\phi$ is
dense in~$V$.  Then $V$ is irreducible.
\end{enumerate}
\end{thm}

According to Proposition~\ref{prop:cupcap}, the corresponding
decomposition operation for ideals is to write ideals as the
intersection of other ideals.  However, for general ideals, the
situation is much more complicated than for varieties.  The situation
simplifies for radical ideals (which are in a one-to-one
correspondence with varieties).  This case is discussed next.  The
general case is summarized afterwards.

\begin{defn}
An ideal $I \subseteq \Bbbk[p_1, \ldots, p_r]$ is \emph{prime} if for
all $f,g\in \Bbbk[p_1, \ldots, p_r]$ with $f\cdot g\in I$, one of the
factors~$f,g$ belongs to~$I$.
\end{defn}
For example, $I:=\langle xy\rangle$ is not prime, because $xy\in I$,
but neither $x\in I$ nor $y\in I$.
\begin{thm}
A variety $V$ is irreducible if and only if~$I(V)$ is prime.
\end{thm}
\begin{defn}
A prime ideal $P \subseteq \Bbbk[p_1, \ldots, p_r]$ is a
\emph{minimal prime} of an ideal $I \subseteq \Bbbk[p_1, \ldots, p_r]$
if and only if $V(P)$ is an irreducible component of~$V(I)$.
\end{defn}
There is also an algebraic definition of the minimal primes, and there
are algorithms to compute the minimal primes.  By definition, the
minimal primes of an ideal encode the irreducible decomposition of the
corresponding variety:
\begin{thm}
\begin{enumerate}
\item Any ideal $I \subseteq \Bbbk[p_1, \ldots, p_r]$ has finitely
many minimal primes~$P_1,\dots,P_k$.
\item The ideal $P_1\cap\dots\cap P_k$ equals the radical of~$I$.
\item The irreducible components of $V(I)$ are $V(P_1),\dots,V(P_k)$.
\end{enumerate}
\end{thm}

If $I$ is not radical, then $P_1\cap\dots\cap P_k\subseteq I$.  In
this case, it is still possible to write $I$ as an intersection of
special ideals (called \emph{primary ideals}) in a way that is
algebraically and geometrically meaningful.  This intersection is
called a \emph{primary decomposition}.  The precise definitions are
omitted, since a primary decomposition often adds little to the
statistical understanding.  However, some works in algebraic
statistics written by algebraists who do care about the differences
between ideals and their radicals use this notation.  The following
result explains how a primary decomposition is related to the minimal
primes.

\begin{thm}
Let $I = I_1\cap\dots\cap I_l$ be a primary decomposition of $I
\subseteq \Bbbk[p_1, \ldots, p_r]$, and let $P_i$ be the radical
of~$I_i$.
\begin{enumerate}
\item $V(I) = V(I_1)\cup V(I_2) \cup\dots\cup V(I_l) = V(P_1)\cup
V(P_2) \cup\dots\cup V(P_l)$.
\item Each $P_i$ is prime.
\item Each minimal prime of $I$ is among the~$P_i$.
\item If $P_i$ is not a minimal prime of~$I$, then there is a minimal
prime $P_j$ of $I$ with $P_j\subset P_i$ (and so $V(P_i)\subset
V(P_j)$).
\end{enumerate}
\end{thm}

\begin{example}
Let $I = \langle xy, xz \rangle \in \Bbbk[x,y,z]$.  The variety $V(I)$
consists of the union of the plane where $x = 0$, and the line where
$y = 0, z = 0$.  Hence
$V(\langle xy, xz \rangle ) = V( \langle x \rangle) \cup V( \langle
y,z \rangle) $ is a decomposition into irreducibles.  This corresponds
to the ideal decomposition
$\langle xy, xz \rangle = \langle x \rangle \cap \langle y,z \rangle$.
\end{example}

The primary decomposition need not be unique.

\begin{example}\label{e:embedded}
The ideal $\langle x^2, xy \rangle$ has several
different primary decompositions, e.g.
\[
\langle x^2, xy \rangle = 
\langle x \rangle \cap \langle x^2, y \rangle  = \langle x \rangle \cap \langle x^2, x+ y \rangle
\]
The variety $V(\langle x^2, xy \rangle)$ equals the line where
$x = 0$, corresponding to the unique minimal prime $\langle x\rangle$.
The non-uniqueness of the primary decomposition is related to the fact
that the variety of the ``extra'' component is a subset of one of the
other components.  This variety (which is superfluous in the
irreducible decomposition) is called an \emph{embedded components}.
This example can be analyzed as follows using the computer algebra
system~\textsc{Macaulay2} \cite{M2}.  First set up a polynomial ring
in the indeterminates $x,y$ with the rational numbers $\mathbb{Q}$ as
the coefficient field.  In \textsc{Macaulay2} it is advisable to work
with $\mathbb{Q}$ rather than $\mathbb{R}$ or~$\mathbb{C}$ since the
arithmetic in $\mathbb{Q}$ can be carried out exactly on a computer.
\begin{verbatim}
i1 : R = QQ[x,y]
o1 = R
o1 : PolynomialRing
\end{verbatim}
  The system reports that it understands \verb,R, as a polynomial
  ring.  The following input makes \textsc{Macaulay2} decompose the
  ideal.  The decomposition is computed over $\mathbb{Q}$, but in this
  case it happens to be valid over any field~$\Bbbk$.
\begin{verbatim}
i2 : primaryDecomposition ideal (x^2, x*y)
                       2
o2 = {ideal x, ideal (x , y)}
\end{verbatim}
  If one is only interested in the irreducible decomposition, the
  command \verb,decompose, returns the minimal primes corresponding to the irreducible
  components, discarding all embedded components:
\begin{verbatim}
i3 : decompose ideal (x^2, x*y)
o3 = {ideal x}
\end{verbatim}
\end{example}

\subsection{Binomial ideals}
\label{sec:binomial-ideals}

This Section ends with a short discussion of binomial ideals and toric
ideals, which make frequent appearance in applications.

\begin{defn}\label{d:binomialIdeal}
A \emph{binomial} is a polynomial $p^u - \lambda p^v$,
$\lambda \in \Bbbk$ with at most two terms.  An ideal $I$ is a
\emph{binomial ideal} if it has a generating set of binomials.  A
binomial ideal that is prime and does not contain any variable is a
\emph{toric ideal}.
\end{defn}

The main reason why it is important whether an ideal is binomial is
that there are dedicated algorithms for binomial ideals that are much
faster than the generic algorithms that work for any ideal
\cite{Eisenbud1996, Dickenstein2010, Kahle2012, Kahle2016}.  Note that
there are some instances of ideals that arise in algebraic statistics
that are not binomial in their natural coordinate systems but become
binomial ideals after a linear change of coordinates
\cite{Sturmfels2005}.

Let $A\in\mathbb{Z}^{h\times r}$ be an integer matrix, and consider
the ideal
\begin{equation*}
I_{A} := \big\langle p^{u_{+}} - p^{u_{-}}  : u = u_{+}-u_{-}\in\ker_{\mathbb{Z}}A \big\rangle
\end{equation*}
in the polynomial ring $\Bbbk[p_{1},\dots,p_{r}]$, where
$u=u_{+}-u_{-}$ is the decomposition of $u$ into its positive and
negative part $u_{+},u_{-}\in\mathbb{N}^{r}$.  Clearly, $I_{A}$ is
binomial and does not contain any of the~$p_i$.  One can also show
that $I_{A}$ is prime, and thus it is an example of a toric ideal.  In
fact, any toric ideal is of this form up to a scaling of
coordinates~\cite[Corollary~2.6]{Eisenbud1996}.  The generating set
above is infinite, but Theorem~3.1 in
\cite{DiaconisSturmfels98:Algebraic_Sampling} shows that finite
generating sets of toric ideals are related to Markov bases, which can
be computed using the software \texttt{4ti2}~\cite{4ti2}.

\subsection{Real algebraic geometry}
\label{sec:real-algebr-geom}

In addition to polynomial equations, in many situations in statistics
it is useful to consider solutions to polynomial inequalities as well.
This is the subject of the field \emph{real algebraic geometry}.
Inequalities only make sense over an ordered field like $\mathbb{R}$
(but not over~$\mathbb{C}$).  For
simplicity, the following definitions and results are formulated
with~$\mathbb{R}$.  Again, this text only contains the basic
definitions.  For more details the reader is referred
to~\cite{Basu2006, Bochnak1998}.

\begin{defn}
Let $\mathcal{F}, \mathcal{G} \subseteq \mathbb{R}[p_1, \ldots, p_r]$
be sets of polynomials with $\mathcal{G}$ finite.  The \emph{basic
semialgebraic set} defined by $\mathcal{F}$ and $ \mathcal{G}$ is
\[
\big\{ a \in \mathbb{R}^r :  f(a) = 0 \mbox{ for all } f \in \mathcal{F} \mbox{ and } g(a) > 0 \mbox{ for all } g \in \mathcal{G}  \big\}.
\]
A \emph{semialgebraic set} is a finite union of basic semialgebraic sets.
\end{defn}

Here are some common examples of semialgebraic sets arising in statistics.

\begin{example}
The open probability simplex 
\[
{\rm int} (\Delta_{r-1} ):=  
\big\{ p \in \mathbb{R}^r :  \sum_{i = 1}^r p_i  =1,  p_i > 0, i = 1, \ldots, r  \big\}
\]
consists of all probability distributions for a categorical random
variables with $r$ states.  It is a basic semialgebraic set: In the
above definition, one may take $\mathcal{F} = \big\{ \sum_{i = 1}^r p_i -1
\big\}$ and $\mathcal{G} = \big\{p_1, \ldots, p_r \big\}$.  The probability
simplex
\[
\Delta_{r-1} :=  
\big\{ p \in \mathbb{R}^r :  \sum_{i = 1}^r p_i  =1,  p_i \geq 0, i = 1, \ldots, r  \big\}
\]
is a semialgebraic set.  It can be written as the union of $2^{r} -1$ basic semialgebraic sets.
\end{example}

\begin{example}
The cone $PD_m$ of $m \times m$ positive definite symmetric matrices
is an example of a basic semialgebraic set in
$\mathbb{R}^{\binom{m+1}{2}}$, where $\mathcal{F} = \emptyset$ and
where $\mathcal{G}$ consists of the principal subdeterminants of an
$m \times m$ symmetric matrix of indeterminates.  For instance, if
$m = 3$ consider the polynomial ring
$\mathbb{R}[ \sigma_{11}, \sigma_{12}, \sigma_{13}, \sigma_{22},
\sigma_{23}, \sigma_{33}]$ and the symmetric matrix of indeterminates
\[
\Sigma = \begin{pmatrix}
\sigma_{11} & \sigma_{12} & \sigma_{13}  \\
\sigma_{12} & \sigma_{22} & \sigma_{23}  \\
\sigma_{13} & \sigma_{23} & \sigma_{33}  \\
\end{pmatrix}.
\]
The symmetry has been enforced by making certain entries in the matrix
equal.  The set of polynomials defining $PD_3$ can be chosen to be
\[
\mathcal{G} =  \big\{  \sigma_{11}, \sigma_{11}\sigma_{22} - \sigma_{12}^2, \det \Sigma \big\},
\]
the set of leading principal minors of $\Sigma$.  The cone of positive
semidefinite symmetric matrices is a semialgebraic set, which can be
realized by using non-strict inequalities with the much larger set of
all principal minors of~$\Sigma$.
\end{example}


\section{Conditional Independence Ideals}\label{sec:ciideals}

This section shows how the algebraic tools introduced in
Section~\ref{sec:alggeom} can be used to analyze conditional
independence structures.  The tools can be applied in the settings of
discrete random variables and jointly normal variables, but in
different ways.

\subsection{Discrete random variables}

Let $X_1, X_2, \ldots, X_m$ be finite discrete random
variables.  Suppose that the state
space of $X_i$ is $[r_i] := \{1,2,\ldots, r_i \}$.  There is an
algebraic description of the set of all distributions that satisfy a
given conditional independence statement.  The first example comes
from the simplest CI statement: $\ind{1}{2}$.

\begin{prop}
  Let $X_1, X_2$ be discrete random variables where the state space of
  $X_i$ is~$[r_i]$.  Let $p_{i_1i_2} = P(X_1 = i_1, X_2 = i_2)$ and
  let $p = (p_{i_1 i_2})_{i_1 \in [r_1], i_2 \in [r_2]}$ be the joint
  probability mass function of $X_1$ and~$X_2$.  Then $\ind{1}{2}$ if
  and only if $p$ is a rank one matrix.
\end{prop}

\begin{proof}
If $\ind{1}{2}$ then
$P(X_1 = i_1, X_2 = i_2) = P(X_1 = i_1) P(X_2 = i_2)$.  This expresses
the joint probability mass function as an outer product of two nonzero
vectors, hence $p$ has rank~one.

Conversely, if $p$ has rank one, it is expressed as the outer product
of two vectors $p = \alpha^T \beta$.  Since $p$ is a matrix of
nonnegative real numbers, one can assume that $\alpha$ and $\beta$ are also nonnegative.
Let $\|.\|_{1}$ denote the $l_{1}$-norm.  Replacing $\alpha$ by
$\alpha/ \| \alpha \|_1$ and $\beta$ by $\beta / \|\beta \|_1$, yields
a rank one factorization for $p$ where the two factors are necessarily
the marginal distributions of $X_1$ and $X_2$ respectively.  Hence
$\ind{1}{2}$.
\end{proof}

A nonzero matrix having rank one is characterized by the vanishing of
all its $2 \times 2$ subdeterminants.  Hence, one can associate an
ideal to the independence statement $\ind{1}{2}$.

\begin{defn}
The \emph{conditional independence ideal} for the statement $\ind{1}{2}$ is
\begin{align*}
  I_{\ind{1}{2}} & =
                   \langle  p_{i_1 i_2} p_{j_1 j_2}  - p_{i_1 j_2} p_{j_1 i_2}  :
                   i_1, j_1 \in [r_1],  i_2, j_2 \in [r_2]  \rangle \\
                 & = \langle 2 \times 2 \text{ subdeterminants of } p  \rangle  \subseteq
                   \mathbb{R}[p_{i_1, i_2} : i_1 \in [r_1], i_2 \in
                   [r_2]].
\end{align*}
\end{defn}

\begin{example}
Let $r_1 = 2$ and $r_2 = 3$.  Then
\[
I_{\ind{1}{2}} = \langle p_{11} p_{22} - p_{12}p_{21}, p_{11} p_{23} -
p_{13}p_{21}, p_{12} p_{23} - p_{13}p_{22} \rangle.
\]
\end{example}

The conditional independence ideal $I_{\ind{1}{2}}$ captures the
algebraic structure of the independence condition.  Although all
probability distributions would satisfy the additional constraint that
$\sum_{i_1 \in [r_1], i_2 \in [r_2]} p_{i_1i_2} -1 = 0$, this trivial
constraint is not included in the conditional independence ideal
because leaving it out tends to simplify certain algebraic
calculations.  For example, without this constraint $I_{\ind{1}{2}}$
is a binomial ideal.

More generally, any conditional independence condition for discrete
random variables can be expressed by similar determinantal
constraints.  This requires a bit of notation.  The determinantal
constraints are written in terms of the entries of the joint
distribution of $X_1,\dots,X_m$.  This is a \emph{tensor} $p =
(p_{i_1,\dots,i_m})_{i_j \in [r_j]}$.

Let $A,B,C \subset [m]$ be disjoint subsets of indices of the random
variables $X_1,\dots,X_m$, and $D = [m] \setminus (A\cup B\cup C)$ the
set of indices appearing in none of~$A,B,C$.  Any such assignment
yields a grouping of indices and random variables.  The random vector
$X_A = (X_j)_{j\in A}$ takes values in $\mathcal{R}_A = \prod_{j\in
A}[r_j]$.  Let $\mathcal{R}_B, \mathcal{R}_C$ and $\mathcal{R}_D$ be
defined analogously.  The grouping allows one to write $p =
(p_{i_A,i_B,i_C,i_D})$ where now $i_A \in \mathcal{R}_A$ and similarly
for $i_B,i_C$, and~$i_D$.  The final notational gadget is the
marginalization of $p$ over $D$.  The entries of this marginal
distribution are indexed by
$\mathcal{R}_A,\mathcal{R}_B,\mathcal{R}_C$ and have entries
\[
p_{i_A,i_B,i_C,+} = \sum_{i_D \in \mathcal{R}_D} p_{i_A,i_B,i_C,i_D}
\]
The $+$ indicates the summation.

\begin{defn}\label{d:generalCIideal}
The conditional independence ideal for the conditional independence
statement $\ind{A}{B}[C]$ is
\begin{equation*}
\begin{split}
I_{\ind{A}{B}[C]} = \Big\langle p_{i_A,i_B,i_C,+}\cdot p_{j_A,j_B,i_C,+} -
p_{i_A,j_B,i_C,+}\cdot p_{j_A,i_B,i_C,+}, \text{ for all } \\
i_A,j_A\in\mathcal{R}_A, i_B,j_B\in\mathcal{R}_B,
i_C\in\mathcal{R}_{C} \Big\rangle
\end{split}
\end{equation*}
\end{defn}

The notation simplifies for \emph{saturated conditional independence
statements}, for which $A\cup B \cup C = [m]$.  With this condition
there is no marginalization, and the defining polynomials of
$I_{\ind{A}{B}[C]}$ are binomials.

\begin{example}\label{e:saturatedCI}
  Consider three binary random variables $X_1,X_2,X_3$.  Let
  $p_{111},\dots, p_{222}$ denote the indeterminates standing for the
  elementary probabilities in the joint distribution.  The conditional
  independence ideal of the statement $\ind{1}{3}[2]$ is
\[
  I_{\ind{1}{3}[2]} = \langle p_{111}p_{212} - p_{211}p_{112},
  p_{121}p_{222} - p_{221}p_{122}\rangle.
\]
The conditional independence ideal of the statement
$\ind{1}{3}$ is
\[
  I_{\ind{1}{3}} = \langle (p_{111} + p_{121})(p_{212} + p_{222}) -
  (p_{112} + p_{122})(p_{211} + p_{221}) \rangle.
\]
\end{example}

\begin{prop}\label{prop:prime}
For any conditional independence statement $\ind{A}{B}[C]$,
the conditional independence ideal $I_{\ind{A}{B}[C]}$ is a prime
ideal and hence $V( I_{\ind{A}{B}[C]})$ is an irreducible variety.
\end{prop}

Proposition~\ref{prop:prime} is a consequence of the fact that general
determinantal ideals are prime (see \cite{Bruns1988}).
Irreducibility of the variety $V( I_{\ind{A}{B}[C]})$ can also be
deduced from the fact that this variety can be parametrized, for
instance, the set of all probability distributions in $V(
I_{\ind{A}{B}[C]})$ can be realized as the set of probability
distributions in a graphical model.

\begin{example}
A strictly positive joint distribution $p$ of binary random variables
$X_1,X_2,X_3$ satisfies $\ind{1}{3}[2]$ if and only if
\begin{equation}
\label{eq:param}
p_{i_1,i_2,i_3} = s_{i_1,i_2}t_{i_2,i_3}
\end{equation}
for some vectors $(s_{i_1,i_2})_{i_1\in [r_1],i_2\in[r_2]}$,
$(t_{i_2,i_3})_{i_2\in[r_2],i_3\in[r_3]}$ (see Section~1.3 of Part~I).
That is, it lies in the undirected graphical model
\[
X_1 - X_2 - X_3.
\]
Since $V(I_{\ind AB[C]})$ is irreducible, any joint distribution
(possibly with zeros) that satisfies $\ind{1}{3}[2]$ lies in the
closure of the undirected graphical model.  In fact, any such joint
distribution has a parametrization of the form~\eqref{eq:param}, where
$s$ or $t$ also may have zeros.
\end{example}

More interesting than just single statements are combinations of two
or more conditional independence statements.  To determine the classes
of distributions satisfying a collection of independence statements
leads to interesting problems in computational algebra.  Such sets are
typically not irreducible varieties and cannot be parametrized with a
single parametrization.  The first task is to break such a set into
components, and to see if those components have natural
interpretations in terms of conditional independence and can be
parametrized.

\begin{defn}
Let
$\mathcal{C} = \{ \ind{A_1}{B_1}[C_1], \ind{A_2}{B_2}[C_2], \ldots \}$
be a set of conditional independence statements for the random
variables $X_1, X_2, \ldots, X_m$.  The \emph{conditional independence
ideal of $\mathcal{C}$} is the sum of the conditional independence
ideals of the elements of~$\mathcal{C}$:
\[
I_\mathcal{C} = I_{\ind{A_1}{B_1}[C_1]} + I_{\ind{A_2}{B_2}[C_2]} +
\cdots.
\]
\end{defn}
Understanding the probability distributions that satisfy~$\mathcal{C}$
can be accomplished by analyzing an irreducible decomposition
of~$V(I_\mathcal{C})$, which can be obtained from a primary
decomposition of~$I_\mathcal{C}$.

\begin{example}\label{e:pdCItoCI}
Let $X_1$, $X_2$, $X_3$ be binary random variables, and consider
$\mathcal{C} = \{\ind{1}{3}[2], \ind{1}{3}\}$.  The conditional
independence ideal $I_\mathcal{C}$ is generated by three polynomials
of degree $2$:
\begin{multline*}
  I_\mathcal{C} = I_{\ind{1}{3}[2]} + I_{\ind{1}{3}} = \langle
  p_{111}p_{212} - p_{112}p_{211}, p_{121}p_{222} - p_{122}p_{221}, \\
  (p_{111} + p_{121})(p_{212} + p_{222}) - (p_{112} + p_{122})(p_{211}
  + p_{221}) \rangle.
\end{multline*}
The following \textsc{Macaulay2} code asks for the primary
decomposition of this ideal over $\mathbb{Q}$.  It can be shown that
the decomposition is the same over $\mathbb{R}$ and~$\mathbb{C}$.
\begin{verbatim}
loadPackage "GraphicalModels"
S = markovRing (2,2,2)
L = {{{1},{3},{2}}, {{1},{3},{}}}
I = conditionalIndependenceIdeal(S,L)
primaryDecomposition I
\end{verbatim}
This code uses the \textsc{GraphicalModels} package of
\textsc{Macaulay2} which implements many convenient functions to work
with graphical and other conditional independence models.  In
particular, it allows to easily set up the polynomial ring with eight
variables $p_{111},\dots,p_{222}$ with \verb'markovRing' and write out
the equations for $I_\mathcal{C}$ with
\verb'conditionalIndependenceIdeal'.  The command
\verb'primaryDecomposition' is a generic \textsc{Macaulay2} command.
The output of this code consists of two ideals which upon inspection
can be recognized as binomial conditional independence ideals
themselves.  The result is
\[
I_\mathcal{C} = I_{ \ind{\{1,2\}}{3}} \cap I_{ \ind{1}{\{2,3\}}}.
\]
According to Section~\ref{sec:alggeom} this implies a decomposition of
varieties
\[
V(I_\mathcal{C}) = V( I_{ \ind{\{1,2\}}{3}}) \cup V(I_{
\ind{1}{\{2,3\}}}).
\]
On the level of probability distributions, this proves the binary case
of Proposition~\ref{prop:example}.
\end{example}

The general situation may be less favorable than that in
Example~\ref{e:pdCItoCI}.  In particular, the components that appear
need not have interpretations in terms of conditional independence.
The appearing ideals also need not be prime ideals (in general they
are only primary) and it is unclear what this algebraic extra
information may reveal about conditional independence.  For examples
on how to extract information from primary decompositions
see~\cite{Herzog2010,Kahle2014b}.

\subsection{Gaussian random variables}
\label{sec:normal-case}

Algebraic approaches to conditional independence can also be applied
to Gaussian random variables.  Let $X \in
\mathbb{R}^m$ be a nonsingular multivariate Gaussian random vector with
mean $\mu\in\mathbb{R}^m$ and covariance matrix $\Sigma \in PD_m$, the
cone of $m \times m$ symmetric positive definite matrices.  One writes
$X \sim \mathcal{N}(\mu, \Sigma)$.  For subsets $A, B \subseteq [m]$
let $\Sigma_{A, B}$ be the submatrix of $\Sigma$ obtained by
extracting rows indexed by $A$ and columns indexed by $B$, that is
$\Sigma_{A,B} = ( \sigma_{a,b})_{a \in A, b \in B}$.

\begin{prop}
Let $X \sim \mathcal{N}(\mu, \Sigma)$ with $\Sigma \in PD_m$.  Let
$A,B, C \subseteq [m]$ be disjoint subsets.  Then the conditional
independence statement $\ind{A}{B}[C]$ holds if and only if the matrix
$\Sigma_{A \cup C, B \cup C}$ has rank~$\leq \#C$.
\end{prop}

A proof of this proposition recognizes $\Sigma_{A \cup C, B\cup C}$ as
a Schur complement of a submatrix of $\Sigma$.  The details can be
found in \cite[Proposition~3.1.13]{drton09:_lectur_algeb_statis}.

Just as the rank one condition on a matrix was characterized by the
vanishing of $2~\times~2$ subdeterminants, higher rank conditions on
matrices can also be characterized by the vanishing of
subdeterminants.  Indeed, a basic fact of linear algebra is that a
matrix has rank $\leq r$ if and only if the determinant of every
$(r +1) \times (r + 1)$ submatrix is zero.  This leads to the
conditional independence ideals for multivariate Gaussian random variables.

Let
$\mathbb{R}[\Sigma] := \mathbb{R}[\sigma_{ij} : 1 \leq i \leq j \leq
m]$ be the polynomial ring with real coefficients in the entries of
the symmetric matrix $\Sigma$.

\begin{defn}
  The \emph{Gaussian conditional independence ideal} for the
  conditional independence statement $\ind{A}{B}[C]$ is the ideal
  \[
    J_{\ind{A}{B}[C]} := \langle \text{$(\#C + 1)$-minors of } \Sigma_{A
      \cup C, B \cup C} \rangle.
  \]
  If
  $\mathcal{C} = \{ \ind{A_1}{B_1}[C_1], \ind{A_2}{B_2}[C_2], \ldots,
  \}$ is a collection of conditional independence statements, the
  Gaussian conditional independence ideal is
  \[
    J_\mathcal{C} =  J_{\ind{A_1}{B_1}[C_1]} +  J_{\ind{A_2}{B_2}[C_2]} +
    \cdots.
  \]
\end{defn}

\begin{remark}
A common criterion in statistics
says that, in fact, $\ind{A}{B}[C]$ holds if and only if
$\det(\Sigma_{\{a\}\cup C, \{b\}\cup C})$ vanishes for all~$a\in A$,
$b\in B$.  Since, by assumption, $C$ is non-singular, it is easy to
see that this condition is, in fact, equivalent to the vanishing of
all $(\#C+1)$-minors of~$\Sigma_{A \cup C, B \cup C}$.
\end{remark}

\begin{example}
  Consider the conditional independence statement
  $\ind{2}{\{1,3\}}[4]$.  The ideal $J_{\ind{2}{\{1,3\}}[4]}$ is
  generated by the $2 \times 2$ minors of the matrix
  \[
    \Sigma_{\{2,4\}, \{1,3,4\} }  = 
    \begin{pmatrix}
    \sigma_{12} & \sigma_{23} & \sigma_{24}  \\
    \sigma_{14} & \sigma_{34} & \sigma_{44}
    \end{pmatrix}.
  \]
  Since $\Sigma$ is a symmetric matrix, $\sigma_{ij} = \sigma_{ji}$
  and one always writes $\sigma_{ij}$ with $i \leq j$.  Then
  \[
    J_{\ind{2}{\{1,3\}}[4]}  = \langle 
    \sigma_{12}\sigma_{34} - \sigma_{14}\sigma_{23},
    \sigma_{12}\sigma_{44} - \sigma_{14}\sigma_{24},
    \sigma_{23}\sigma_{44} - \sigma_{34}\sigma_{24} \rangle.
  \]
\end{example}

\begin{example}
Let $X_1$, $X_2$, $X_3$ be jointly Gaussian random variables.  The
conditional independence ideal of
$\mathcal{C} = \{\ind{1}{3}[2], \ind{1}{3}\}$ is
\[
J_\mathcal{C} = J_{\ind{1}{3}[2]} + J_{\ind{1}{3}} = \langle
\sigma_{13} \sigma_{22} - \sigma_{12} \sigma_{23}, \sigma_{13}
\rangle.
\]
Straightforward manipulations of these ideals show 
\[
J_\mathcal{C} = \langle  \sigma_{13} \sigma_{22} - \sigma_{12} \sigma_{23}, \sigma_{13}  \rangle = 
\langle \sigma_{12} \sigma_{23}, \sigma_{13}  \rangle
=  \langle \sigma_{12}, \sigma_{13} \rangle \cap \langle \sigma_{23}, \sigma_{13} \rangle 
\]
\[
=  J_{ \ind{\{1,2\}}{3}}  \cap J_{ \ind{1}{\{2,3\}}}.
\]
This last primary decomposition proves the Gaussian case of
Proposition~\ref{prop:example}.
\end{example}

\subsection{The contraction axiom}
\label{sec:contraction}

When computing the decompositions of conditional independence ideals,
there might be components that are ``uninteresting" from the
statistical standpoint.  These components might not intersect the
region of interest in probabilistic applications (e.g.~they might miss
the probability simplex or the cone of positive definite matrices) or
they might have non-trivial intersections but that intersection is
contained in some other component.

\begin{example} Let $X = (X_1, X_2, X_3)$ be a multivariate Gaussian
random vector.  The conditional independence ideal of
$\mathcal{C} = \{ \ind{1}{2}[3], \ind{2}{3} \}$ is
\[
J_\mathcal{C}  =  \langle \sigma_{12}\sigma_{33} - \sigma_{13}\sigma_{23} , \sigma_{23} \rangle
\]
which has primary decomposition
\[
J_\mathcal{C} = \langle \sigma_{12}\sigma_{33} , \sigma_{23} \rangle =
\langle \sigma_{12} , \sigma_{23} \rangle \cap \langle \sigma_{33} ,
\sigma_{23} \rangle.
\]
This decomposition shows that
\[
V(J_\mathcal{C}) =  V(\langle \sigma_{12}  , \sigma_{23} \rangle) \cup 
V(\langle \sigma_{33}  , \sigma_{23} \rangle).
\]
However, the second component does not intersect the positive definite
cone, because $\sigma_{33}> 0$ for all $\Sigma \in PD_3$.  The first
component is the conditional independence ideal $J_{ \ind{1,3}{2}}$.
From this decomposition it is visible that $\ind{1}{2}[3]$ and
$\ind{2}{3}$ imply that $\ind{1,3}{2}$.  This implication is called
the \emph{contraction axiom}.  See Section~1.5 of Part~I for other CI
axioms.
\end{example}

The contraction axiom also holds for non-Gaussian random variables.
For discrete random variables, it can again be checked algebraically.
The primary decomposition associated to the discrete contraction axiom
is worked out in detail in \cite{Garcia2005}.  The next examples
discusses the binary case as an illustration:
\begin{example}
Let $X_1, X_2, X_3$ be binary random variables.  The conditional
independence ideal of $\mathcal{C} = \{ \ind{1}{2}[3], \ind{2}{3} \}$
is
\begin{align*}
  I_\mathcal{C} &= \big\langle p_{111}p_{221} - p_{121}p_{211},
                    p_{112}p_{222} - p_{122}p_{212}, \\
  & \qquad (p_{111} + p_{211}) (p_{122} + p_{222})  - (p_{112} + p_{212}) (p_{121} + p_{221}) 
  \big\rangle
\end{align*}
which has primary decomposition
\begin{align*}
I_\mathcal{C}   =  I_{\ind{1,3}{2}}  & \cap 
\langle p_{122} + p_{222}, p_{112} + p_{212}, p_{121} p_{211} - p_{111}p_{221}
                 \rangle \\
  & \cap
\langle   
p_{121} + p_{221}, p_{111} + p_{211}, p_{122} p_{212} - p_{112}p_{222}
\rangle.
\end{align*}
The intersection of the second component with the probability simplex
forces that $p_{122} = p_{222} = p_{112} = p_{212} = 0$.  This in turn
implies that
\[
V(\langle   
p_{122} + p_{222}, p_{112} + p_{212}, p_{121} p_{211} - p_{111}p_{221}
\rangle)  \cap  \Delta_7   \subseteq V( I_{\ind{1,3}{2}})
\]
A similar argument holds for the third component.  So although the
variety $V(I_\mathcal{C})$ has three components, only one of them is
statistically meaningful:
\[
V(I_\mathcal{C}) \cap \Delta_7  =  V(I_{\ind{1,3}{2}}) \cap  \Delta_7 .
\]
\end{example}


\section{Examples of Decompositions of Conditional Independence Ideals}
\label{sec:examples}

This section studies some examples of families of conditional
independence statements and how algebraic tools can be used to
understand them.  The first example is a detailed study of the
intersection axiom, and the second example concerns the conditional
independence statements associated to the $4$-cycle graph.

\subsection{The intersection axiom}
\label{sec:intersection-axiom}

The \emph{intersection axiom} (see {\bf  Section 1.5})  is the following implication of
conditional independence statements:
\begin{equation}\label{eq:IntersectionAxiom}
\ind{A}{B}[C\cup D] \text{ and } \ind{A}{C}[B \cup D]
\Rightarrow
\ind{A}{B\cup C}[D].
\end{equation}
This implication is valid for strictly positive probability
distributions.  Algebraic techniques can be used to study how its
validity extends beyond this.

The question about the primary decomposition of the ideal(s)
$I_{\{\ind{A}{B}[C\cup D],\ind{A}{C}[B \cup D]\}}$ was first asked
in~\cite[Chapter~6.6]{drton09:_lectur_algeb_statis}, and the answer is
due to~\cite{Fink2011}.  Grouping variables if necessary, one can
assume $A=\{1\}$, $B=\{2\}$, $C=\{3\}$.  Moreover, let~$D=\emptyset$.
From this one can always recover the general case by adding
conditioning constraints.

\begin{prop}[Proposition~1 in \cite{Fink2011}]
\label{prop:CEF-theorem}
The ideal $I_{\{\ind{X_1}{X_2}[X_3],\ind{X_1}{X_3}[X_2]\}}$ is
radical, that is, its irredundant primary decompositions consists only
of prime ideals.  These minimal primes correspond to pairs of
partitions $[r_2]=A_1\cup\dots\cup A_s$, $[r_3]=B_1\cup\dots\cup B_s$
of the same size.  The minimal prime $P$ corresponding to two
partitions is
\begin{multline*}
P = \big\langle p_{i_1i_2i_3} : i_1\in [r_1], i_2\in A_j, i_3\in B_k\text{ for some } j\neq k \big\rangle
\\
 + \big\langle p_{i_1i_2i_3}p_{i'_1i'_2i'_3} -
 p_{i_1i'_2i'_3}p_{i'_1i_2i_3} : i_1,i'_1\in [r_1], i_2,i_2'\in A_j, i_3,i_3'\in B_j\text{ for some }j\big\rangle.
\end{multline*}
\end{prop}
The paper \cite{Fink2011} uses a different formulation in terms of
complete bipartite graphs.  It can be seen that our formualation is
equivalent.

To give a statistical interpretation to
Proposition~\ref{prop:CEF-theorem}, whenever the joint distribution of
$X_1,X_2,X_3$ lies in the prime~$P$ corresponding to the two
partitions $[r_2]=A_1\cup\dots\cup A_s$, $[r_3]=B_1\cup\dots\cup B_s$,
construct a random variable $B$ as follows: put $B := j$ whenever
$X_2\in A_j$ and~$X_3\in B_j$.  Thus, $B$ is uniquely defined except
on a set of measure zero, since $P(X_2\in A_j,X_3\in B_k)=0$ for
$j\neq k$, which follows from the containment of monomials in~$P$.
The variable~$B$ specifies in which blocks of the two partitions the
random variables $X_2$ and~$X_3$ lie.  Now the binomials in~$P$ imply
that $\ind{X_1}{\{X_2,X_3\}}[B]$.
\begin{cor}
\label{cor:intersection}
Suppose that $X_1,X_2,X_3$ satisfy $\ind{X_1}{X_2}[X_3]$ and
$\ind{X_1}{X_3}[X_2]$.  Then there is a random variable~$B$ that
satisfies:
\begin{enumerate}
\item $B$ is a (deterministic) function of~$X_2$;
\item $B$ is a (deterministic) function of~$X_3$;
\item $\ind{X_1}{\{X_2,X_3\}}[B]$.
\end{enumerate}
Conversely, whenever there exists a random variable~$B$ with
properties 1.\ to 3., the random variables $X_1,X_2,X_3$ satisfy
$\ind{X_1}{X_2}[X_3]$ and $\ind{X_1}{X_3}[X_2]$.
\end{cor}
The case where $B$ is a constant corresponds to the CI statement
$\ind{X_1}{\{X_2,X_3\}}$.  The intersection axiom can be recovered by
noting that a function $B$ that is a function of only $X_{2}$ as well
as a function of only $X_{3}$ is necessarily constant, if the joint
distribution of $X_{2}$ and $X_{3}$ is strictly positive.

Similar results hold for all families of CI statements of the form
$\ind{A}{B_i}[C_i]$, where $A\cup B_i\cup C_i=V$ and where $A$ is
fixed for all statements,
see~\cite{Rauh13:Binomial_edge_ideals,RauhAy14:Robustness_and_systems_design}
(the case where $X_A$ is binary was already described
in~\cite{Herzog2010}).  The corresponding CI ideal is still radical,
and the minimal primes have a similar interpretation.  However,
finding the minimal primes is more difficult and involves solving a
combinatorial problem.
Finally, Corollary~\ref{cor:intersection} can be generalized to
continuous random variables \cite{peters2015intersection}.

\subsection{The four-cycle}
\label{sec:C4}

Consider four discrete random variables $X_1,X_2,X_3,X_4$ and the
(undirected) graphical model of the four cycle $C_4=(V,E)$ with edge
set $E=\{(1,2),(2,3),(3,4),(1,4)\}$.  The global Markov CI statements
of this graph are
\begin{equation*}
\operatorname{global}(C_4) = \{
\ind{1}{3}[\{2,4\}]
,\;
\ind{2}{4}[\{1,3\}] \}.
\end{equation*}
The primary decomposition of the corresponding CI ideal was studied
in~\cite{Kahle2014b} in the case where $X_1$ and $X_3$ are binary.
The case that all variables are binary is as follows:
\begin{prop}[Theorem~5.6 in \cite{Kahle2014b}]
The minimal primes of the CI ideal $I_{\operatorname{global}(C_4)}$ of
the binary four cycle are the toric ideal $I_{C_4}$ and the monomial
ideals
\begin{equation*}
P_i = \langle p_x : x_i = x_{i+1} \rangle,
\qquad
P'_i = \langle p_x : x_i \neq x_{i+1} \rangle
\quad\text{for $1\le i<4$.}
\end{equation*}
\end{prop}
The ideal $I_{C_4}$ equals the vanishing ideal of the graphical model,
to be discussed in Section~\ref{sec:vanishing}.  Interestingly, in
this primary decomposition, all ideals except $I_{C_4}$ are monomial
ideals; that is, they only give support restrictions on the
probability distribution.

The primary decomposition of the CI ideal gives an irreducible
decomposition of the corresponding set of probability distributions.
This leads to the statement of Proposition~\ref{prop:binary-C4} in the
introduction.

When $X_2$ and $X_4$ are not binary, the decomposition of
$I_{\operatorname{global}(C_4)}$ involves prime ideals parametrized by
$i\in\{2,4\}$ and two sets $\emptyset \neq C,D \subsetneq [r_i]$.  For
such a choice of $i,C,D$, let $j$ denote the element of
$\{2,4\}\setminus\{i\}$, and let
\begin{multline*}
P_{i,C,D} = \langle p_{1x_21x_4} : x_i \in C, x_{j} \in [r_{j}] \rangle
+ \langle p_{1x_22x_4} : x_i \in D, x_{j} \in [r_{j}] \rangle \\
+ \langle p_{2x_21x_4} : x_i \notin D, x_{j} \in [r_{j}] \rangle +
\langle p_{2x_22x_4} : x_i \notin C, x_{j} \in [r_{j}] \rangle
+ I_{\operatorname{global}(C_4)}
\end{multline*}
the result is the following:
\begin{prop}[Theorem~6.5 in~\cite{Kahle2014b}]
Let $X_1$ and $X_3$ be binary random variables.  The minimal primes of
the CI ideal $I_{\operatorname{global}(C_4)}$ of the four cycle are
the toric ideal $I_{C_4}$ and the ideals $P_{i,C,D}$ for $i\in\{2,4\}$
and $\emptyset \neq C,D\subsetneq [r_i]$.  Furthermore the ideal is
radical and thus equals the intersection of its minimal primes.
\end{prop}
In this case, the non-toric primes are not monomial, but consist of
monomials and binomials.  This fact is independent of the
field~$\Bbbk$.  The following is one example of the kind of
information that can be extracted from knowing the minimal primes of a
conditional independence ideal.
\begin{cor}
Let $X_1,X_2,X_3,X_4$ be finite random variables that
satisfy~$\operatorname{global}(C_4)$, and suppose that $X_1$ and $X_3$
are binary.  Then one (or more) of the following statements is true:
\begin{enumerate}
\item The joint distribution lies in the closure of the graphical model.
\item There is $i\in\{2,4\}$ and there are sets $E,F\subseteq[r_i]$
such that the following holds:
\begin{equation*}
\text{If }(X_1,X_3)=
\left\{
\begin{matrix}
(1,1)\\
(1,2)\\
(2,1)\\
(2,2)
\end{matrix}
\right\},
\text{ then }
\left\{
\begin{matrix}
X_i\in E \\
X_i\in F \\
X_i\notin F \\
X_i\notin E
\end{matrix}
\right\}.
\end{equation*}
\end{enumerate}
Conversely, any probability distribution that satisfies one of these
statements and that satisfies $\ind{2}{4}[\{1,3\}]$ also
satisfies~$\operatorname{global}(C_4)$.
\end{cor}


\section{The Vanishing Ideal of a Graphical Model}\label{sec:vanishing}

Graphical models can be represented via either parametric descriptions
(e.g.~factorizations of the density function) or implicit descriptions
(e.g.~Markov properties and conditional independence constraints).
One use of the algebraic perspective on graphical models is to find
the complete implicit description of the model, in particular, to find
the vanishing ideal of the model.  As described in
Definition~\ref{def:vanishing}, the vanishing ideal of a set $S$ is
the set of all polynomial functions that evaluate to zero at every
point in $S$.  Although some graphical models have complete
descriptions only in terms of conditional independence constraints,
understanding the vanishing ideal can be useful for more complex
models or hidden variable models where conditional independence is not
sufficient to describe the model, for instance, the mixed graph models
studied in Section~2 of Part~I.

\begin{example}
Consider the four cycle $C_4$ and let $X_1, X_2, X_3, X_4$ be binary
random variables.  The vanishing ideal
$I_{C_4} \subseteq \mathbb{R}[p_{i_1i_2i_3i_4}: i_1,i_2,i_3,i_4 \in
\{1,2\}]$ is generated by $16$ binomials, $8$ of which have degree $2$
and $8$ of which have degree $4$.  The degree $2$ binomials are all
implied by the two conditional independence statements
$\ind{1}{3}[\{2,4\}]$ and $\ind{2}{4}[\{1,3\}]$.  On the other hand,
the degree $4$ binomials are not implied by the conditional
independence constraints, even when we restrict to probability
distributions.  One example degree four polynomial is
\[
p_{1111}p_{1222}p_{2122}p_{2211} - p_{1122}p_{1211}p_{2111}p_{2222}
\]
and the others are obtained by applying the symmetry group of the
four cycle and permuting levels of the random variables.
\end{example}

Even in the simple example of the four cycle, there are generators of
the vanishing ideal that do not correspond to conditional independence
statements.  It seems an important problem to try to understand what
other types of equations can arise.
Theorem~3.2 in~\cite{Geiger2006} shows that the vanishing ideal of a
graphical model of discrete random variables is the toric
ideal~$I_{A}$, where $A$ is the design matrix of the graphical model,
defined in the end of Section~\ref{sec:alggeom}.  A classification for
discrete random variables and undirected graphs of when no further
polynomials are needed beyond conditional independence constraints is
obtained in the following theorem:

\begin{thm}[Theorem~4.4 in~\cite{Geiger2006}] \label{thm:Geiger}
Let $G$ be an undirected graph and let $\mathcal{M}_G$ be its graphical model  for
discrete random variables.  Then the vanishing ideal
$I(\mathcal{M}_G)$ is equal to the conditional independence ideal
$I_{\operatorname{global}(G)}$ if and only if $G$ is a chordal graph.
\end{thm}

It is unknown what the appropriate analogue of Theorem~\ref{thm:Geiger}
for other families of graphical models, either with different 
classes of graphs (e.g.~DAGs) or with other types of random variables (e.g.~Gaussian).
Computational studies of the vanishing ideals appear in many 
different papers:  for Bayesian networks with discrete random variables \cite{Garcia2005},
for Bayesian networks with Gaussian random variables \cite{Sullivant2008},
for undirected graphical models with Gaussian random variables \cite{Sturmfels2010}. 
A characterization for which graph families the vanishing
ideal is equal to the conditional independence ideal of global
Markov statements is lacking in all these cases.

One natural question is to determine the other generators of the
vanishing ideal that do not come from conditional independence, and to
give combinatorial structures in the underlying graphs that imply that
these more general constraints hold.  For instance, for mixed Gaussian
graphical models, conditional independence constraints are
determinantal, but not every determinantal constraint comes from a
conditional independence statement, and there is a characterization of
which determinantal constraints come from conditional
independence:

A mixed graph $G = ([m],B,D)$ is a triple where
$[m] = \{1,2, \ldots, m\}$ is the vertex set, $B$ is a set of
unordered pairs of $[m]$ representing bidirected edges in $G$, and $D$
is a set of ordered pairs of $[m]$ represented directed edges in~$G$.
There might also be both directed and bidirected edges between a pair
of vertices.  To the set of $B$ of bidirected edges one associates the
set of symmetric positive definite matrices
\[
PD(B)  =   \{  \Omega \in PD_m : \omega_{ij} = 0 \mbox{ if } i \neq j \mbox{ and }
 i \leftrightarrow j \notin B  \}.
\]
To the set of directed edges one associates the set of $m \times m$
matrices
\[
\mathbb{R}^D  =   \{  \Lambda \in \mathbb{R}^{m \times m} :  \lambda_{ij} = 0 \mbox{ if } 
i \to j  \notin D  \}.
\]
Let $\epsilon \sim \mathcal{N}(0, \Omega)$ and let $X$ be a jointly normal
random vector satisfying the structural equation system
\[
X = \Lambda^T X   +  \epsilon.
\]
\newcommand{\Id}{\mathrm{Id}} This is an example of a linear
structural equation model, and contains as special cases various
families of graphical models.  Let $\mathrm{Id}$ denote the
$m\times m$ identity matrix.  With these assumptions, if
$(\Id-\Lambda)$ is invertible,
\[
X \sim \mathcal{N}(0, \Sigma) \mbox{ where }  
\Sigma = (\mathrm{Id} - \Lambda)^{-T} \Omega( \mathrm{Id} - \Lambda)^{-1}.
\]

\begin{figure}
\centering
\resizebox{4cm}{!}{\includegraphics{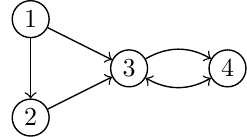}}
\caption{The mixed graph from Section~\ref{sec:vanishing}\label{fig:trek}}
\end{figure}
\begin{example}
Consider the mixed graph $G$ from Figure~\ref{fig:trek}.
In this case, $PD(B)$ is the set of positive definite
matrices of the form
\[
\Omega = \begin{pmatrix}
\omega_{11} & 0 & 0 & 0 \\
0 & \omega_{22} & 0 & 0 \\
0 & 0 & \omega_{33} & \omega_{34}  \\
0 & 0 & \omega_{34}  & \omega_{44}
\end{pmatrix}
\]
and $\mathbb{R}^D$ is the set of real matrices of the form
\[
\Lambda = 
\begin{pmatrix}
0 & \lambda_{12} & \lambda_{13} & 0 \\
0 & 0 & \lambda_{23} &  0 \\
0 & 0 & 0 & \lambda_{34} \\
0 & 0 & 0 & 0 
\end{pmatrix}.
\]
A positive definite matrix $\Sigma$ belongs to the graphical model associated
to this mixed graph, if and only if there are $\Omega \in PD(B)$ and $\Lambda \in \mathbb{R}^D$
such that $\Sigma = (\mathrm{Id} - \Lambda)^{-T}  \Omega (\mathrm{Id} - \Lambda)^{-1}$.
\end{example}

\begin{defn}
  Let $G = ([m], B,D)$ be a mixed graph.  A \emph{trek} between
  vertices $i$ and $j$ in $G$ consists of either
\begin{itemize}
\item a pair $(P_L, P_R)$ where $P_L$ is a directed path ending in $i$ and $P_R$ is a directed path 
  ending in $j$ where both $P_L$ and $P_R$ have the same source, or
\item a pair $(P_L, P_R)$ where $P_L$ is a directed path  ending in $i$ and $P_R$ is a directed path 
  ending in $j$ such that the source of $P_L$ and the source of $P_R$
  are connected by a bidirected edge.
\end{itemize}
Let $\mathcal{T}(i,j)$ denote the set of all treks in $G$ between $i$
and $j$.
\end{defn}

To each trek $T = (P_L, P_R)$ one associates the trek monomial $m_T$
which is the product with multiplicities of all $\lambda_{st}$ over
all directed edges appearing in $T$ times $\omega_{st}$ where $s$ and
$t$ are the sources of $P_L$ and $P_R$.  One reason for the interest
in treks is the \emph{trek rule}, which says the for the Gaussian
graphical model associated to $G$
\[
\sigma_{ij}  =  \sum_{T \in \mathcal{T}(i,j)}  m_T.
\]
For instance, for the mixed graph in Figure \ref{fig:trek}, the pair $(\{1 \to 2, 2 \to 3\}, \{1 \to 3\} )$
is a trek from $3$ to $3$.  
The corresponding trek monomial $m_T$ is $\omega_{11} \lambda_{12} \lambda_{23} \lambda_{13}$. 

\begin{defn}
Let $A, B, C_A, C_B$ be four sets of vertices of $G$, not necessarily disjoint. 
The pair of sets $(C_A, C_B)$ \emph{t-separates} (short for \emph{trek separates}) $A$ from $B$ if for every $a \in A$
and $b \in B$ and every
trek $(P_L, P_R) \in \mathcal{T}(a,b)$, either $P_L$ has a vertex in $C_A$ or
$P_R$ has a vertex in $C_B$, or both.
\end{defn}

\begin{thm}\cite{Sullivant2010}
Let $G = ([m], B,D)$ be a mixed graph and $A$ and $B$ two subsets
of $[m]$ with $|A| = |B| = k$.  Then the minor
$\det \Sigma_{A,B}$ belongs to the vanishing ideal $I_G$
if and only if there is a pair of sets $C_A, C_B$ such that $(C_A, C_B)$
t-separates $A$ and $B$ and such that $|C_A| + |C_B| < k$.
\end{thm}

The t-separation criterion can produce implicit constraints for
structural equation models in situations where there are no
conditional independence constraints.

\begin{example}
\label{ex:trek}
Consider the mixed graph $G$ from Figure~\ref{fig:trek}.  The
vanishing ideal of the model is
$I_G = \langle \det \Sigma_{\{1,2\}, \{3,4\}} \rangle$.  This
determinantal constraint is not a conditional independence constraint.
It is implied by the t-separation criterion because the pair
$( \emptyset, \{3\})$ t-separates $\{1,2\}$ and $\{3,4\}$.
\end{example}

\begin{remark}
In the case where there are hidden random variables, the vanishing
ideal is typically not sufficient to completely describe the set of
probability distributions that come from the graphical model.  Usually
one also needs to consider inequality constraints and other
semialgebraic conditions.  This problem is discussed in more detail in
\cite{Allman2014, Allman2015, Zwiernik2016}, among others.
\end{remark}


\section{Further Reading}\label{sec:further}

Diaconis, Eisenbud and Sturmfels \cite{Diaconis1998} were the first to
consider primary decompositions for statistical applications, in
particular the analysis of the connectivity of certain random walks.
This perspective was also picked up in \cite{Kahle2014b} using
conditional independence ideals.  Primary decomposition of conditional
independence ideals also makes an appearance in the following papers
not already mentioned \cite{Drton2010, Fink2015, Kirkup2007,
Sullivant2009,Swanson2013}.

The algebraic view on undirected graphical models was presented
in~\cite{Geiger2006}, which began extensive study of the vanishing
ideals of undirected graphical models for discrete random variables.
Focus has been on developing techniques for constructing generating sets
of the vanishing ideals with \cite{Engstrom2014, Hosten2002, Rauh2016}
being representative papers in this area.  Vanishing ideals of
undirected models with Gaussian random variables and models for DAGs
have not been much studied.  Some papers that initiated their study
include \cite{Garcia2005,Sturmfels2010, Sullivant2008}.

\bibliographystyle{amsplain}
\bibliography{math}

\end{document}